\documentclass[11pt,twoside]{article}
\usepackage{mathrsfs,amsfonts,amsmath,amssymb}
\usepackage{url}

\usepackage{comment} 
\usepackage{latexsym}
\usepackage{enumitem}

\newtheorem{satz}{Theorem}
\newtheorem{proposition}[satz]{Proposition}
\newtheorem{theorem}[satz]{Theorem}
\newtheorem{lemma}[satz]{Lemma}

\newtheorem{corollary}[satz]{Corollary}
\newtheorem{remark}[satz]{Remark}
\newtheorem{exm}[satz]{Example}

\def\Z{\mathbb {Z}}
\def\F{\mathbb {F}}
\def\E{\mathsf{E}}

\def\a{\alpha}

\def\d{\delta}

\def\({\big (}
\def\){\big )}

\def\dim{{\rm dim}}
\def\le{\leqslant}
\def\ge{\geqslant}
\def\_phi{\varphi}
\def\eps{\varepsilon}

\def\Gr{{\mathbf G}}

\def\la{\lambda}

\def\T{\mathsf{T}}

\def\F{\mathbb {F}}
\def\R{{\mathbb R}}

\def\bp{\bigskip}

\author{Shkredov I.D.}
\title{On common energies and sumsets  
}
\date{}
\begin{document}
	\maketitle

\begin{center}
	Annotation.
\end{center}

{\it \small
    We obtain a polynomial criterion for a set to have a small doubling  in terms of the common energy of its subsets. 
}
\\

\section{Introduction}
\label{sec:intr}

Given an abelian group $\Gr$ and two sets $A,B\subseteq \Gr$, define  
%the \textit{product set} (
the {\it sumset} 
%in the abelian case) 
of $A$ and $B$ as 
\begin{equation}\label{def:A+B_intr}
    A+B:=\{a+b ~:~ a\in{A},\,b\in{B}\}\,.
\end{equation}
In a similar way we define the {\it difference sets} and the {\it higher sumsets}, e.g., $2A-A$ is $A+A-A$. 
The {\it doubling constant} of a finite set $A$ is defined by formula 
\begin{equation}\label{def:doubling}
    \mathcal{D} [A] := \frac{|A+A|}{|A|} 
\end{equation}
and is an important additive--combinatorial characteristic of the set $A$. 
The study of the structure of sumsets is a fundamental  problem in  classical additive combinatorics \cite{TV}, and this relatively new branch of mathematics provides  us with a number of excellent results on this object. 
For example, Freiman's theory about the structure of the family of sets having small doubling (i.e.,  sets for which the ratio \eqref{def:doubling} is ``bounded'' with respect to  the size of $A$) seems especially complete (see, e.g.,  \cite{Freiman_book}, \cite[Section 5]{TV}  and recent results in \cite{GGMT_Marton_bounded}).
Another combinatorial concept closely related to $\mathcal{D} [A]$ is the {\it additive energy} $\E(A,A)$ or, more generally, the {\it common additive energy} $\E(A,B)$ of $A$ and $B$, which is defined as
\begin{equation}\label{def:common_energy_intr}
 \E^{} (A,B) = |\{ (a_1,a_2,b_1,b_2) \in A\times A \times B \times B ~:~ a_1 - b^{}_1 = a_2 - b^{}_2 \}| \,.
\end{equation}
A trivial application of the Cauchy--Schwarz inequality gives us the connection between the sumset of $A$ and $B$ and its  common energy
\begin{equation}\label{f:energy_CS}
    \E^{} (A,B) |A \pm B| \ge |A|^2 |B|^2 \,,
\end{equation}
but a much deeper connection is given to us by the famous Balog--Szemer\'edi--Gowers theorem (see \cite{BSz_statistical} and \cite{Gowers_4}). 

\begin{theorem}
    Let $\Gr$ be an abelian group, $A\subseteq \Gr$ be a set and $K\ge 1$ be a real number. 
    Then 
\begin{equation}\label{f:BSzG_intr}
    \E(A,A) \gg \frac{|A|^3}{K^{C_1}} \quad \quad 
    \mbox{ iff }
        \quad \quad 
    \exists A' \subseteq A ~:~ |A'+A'| \ll K^{C_2} |A| 
    \quad 
    %\quad 
    \mbox{ and }
    \quad
    %\quad 
    |A'| \gg \frac{|A|}{K^{C_2}} \,.
\end{equation}
    Here $C_1 = O(C_2)$ and $C_2 = O(C_1)$. 
\label{t:BSzG_intr}
\end{theorem}

    In other words, the dependence  between the quantities $K^{C_1}$, $K^{C_2}$ in the criterion \eqref{f:BSzG_intr} is {\it polynomial}, and this  partially explains the importance of this remarkable result. On the other hand, the Balog--Szemer\'edi--Gowers theorem is formulated in terms of {\it subsets} of $A$, and, as is known,
    %it is well--known  that 
    it is impossible to do without them. 
    %it  one cannot avoid using of this language. 

    In this paper we connect the two basic quantities \eqref{def:doubling} and \eqref{def:common_energy_intr} in another way and we give a polynomial criterion for a set to have a small doubling. 
In particular, we obtain (the complete formulation can be found in our main Theorem \ref{t:equiv}) 
for any $A\subseteq \Gr$ and an arbitrary real number $K\ge 1$  that 
\begin{equation}\label{eq:intr}
    |A+A| \ll K^{C_1} |A| 
        \quad \quad 
    \mbox{ iff }
        \quad \quad 
    \forall X,Y \subseteq A,\, 
    %\forall Y\subseteq A, 
    |X|\ge |A|/2 ~:~ \E(X,Y) \ge 
    %\frac{|X|^2 |Y|^2}{K^{C_2} |A|} 
    \frac{|X| |Y|^2}{K^{C_2}} 
    %\,. 
\end{equation}
and therefore it is possible to express the doubling constant of $A$ in terms of the common additive energy  of subsets of $A$. 
 Here, as above, $C_1 = O(C_2)$ and $C_2 = O(C_1)$.
% as above. 
One direction of the criterion  \eqref{eq:intr}  follows easily from \eqref{f:energy_CS}, but another one is more interesting. 
We also show that the constant $1/2$ in formula \eqref{eq:intr} is optimal and cannot be replaced by $1/2+\eps$ for any $\eps >0$, see Remark \ref{r:1/2-kappa} of Section \ref{sec:proof}.

Another remarkable result of additive combinatorics is the following Ruzsa's triangle inequality,  valid 
%that takes place 
for any sets $A,B,C \subseteq \Gr$ 
\begin{equation}\label{f:triangle_intr}
    |A| |B-C| \le |B-A||A-C| \,.
\end{equation}
Using the common energies and other characteristics that are polynomially  equivalent to the doubling constants of our sets,  we show (see the last formula \eqref{f:triangle_E++} in Theorem  \ref{t:triangle_S} below) 
%or just estimate \eqref{f:E_triangle}) 
that a common energy analogue of \eqref{f:triangle_intr} takes place, namely, (all required definitions can be found in Section \ref{sec:def})  
%, namely, (see the last formula of Proposition \ref{p:triangle_S} below) 
\begin{equation}\label{f:E_triangle}
    \mathcal{E} [B;C] \le \mathcal{E} [B;A] \mathcal{E} [A;C] \,. 
\end{equation}

We hope that the tools developed in this paper will add flexibility to working with the doubling constants.

\section{Definitions}
\label{sec:def}

Below $\Gr$ denotes an abelian group with the group operation $+$. 
The sumset of two sets $A,B \subseteq \Gr$ was defined in formula \eqref{def:A+B_intr} of the introduction. 
We say that the sum of $A$ and $B$ is direct 
%(and we write this fact 
if $|A+B|= |A||B|$.
In this case we sometimes write $A\dotplus B$. 
The important Pl\"unnecke--Ruzsa inequality (see, e.g., \cite{Ruzsa_Plun} or \cite{TV}) says that for any positive integers $n$ and $m$ 
%one has 
the following holds 
\begin{equation}\label{f:Pl-R} 
    |nA-mA| \le \left( \frac{|A+A|}{|A|} \right)^{n+m} \cdot |A| \,.
\end{equation} 
%Further if $|A+B|\le K|A|$ for some sets $A,B \subseteq \Gr$, then for any $n$ one has 
%\begin{equation}\label{f:Pl-R+} 
%    |nB| \le K^n |A| \,.
%\end{equation}
%Thus, 
Therefore, 
bound \eqref{f:Pl-R} connects the cardinality of the %initial 
original 
set $A$ and its higher sumsets $nA-mA$. 
The common  additive energy was defined in \eqref{def:common_energy_intr} and the generalization of the additive energy is as follows 
\begin{equation}\label{def:T_k_intr}
%    \T_k^{} (A) := |\{ (a_1, \dots, a_k,a'_1, \dots, a'_k) \in A^{2k} ~:~ a_1 + \dots + a_k = a'_1 + \dots + a'_k \} |\,.
\T_k (f) = \sum_{a_1 + \dots + a_k = a'_1 + \dots + a'_k} f(a_1) \dots f(a_k) f(a'_1) \dots f(a'_k) \,,
\end{equation}
where $f:\Gr \to \R$ is an arbitrary  function.
We use the same capital letter to denote a set $A\subseteq \Gr$ and   its characteristic function $A: \Gr \to \{0,1 \}$. 
Thus, $\T_2 (A) = \E(A,A)$ and for brevity we will write $\E(A)$ instead of $\E(A,A)$.
%Thus, $\T_2 (A) = \E(A,A)$ and brevity we write $\E(A)$ for $\E(A,A)$.  
Also, we  use representation function notations like  $r_{A+B} (x)$ or $r_{A-B} (x)$ and so on, which counts the number of ways $x \in \Gr$ can be expressed as a sum $a+b$ or  $a-b$ with $a\in A$, $b\in B$, respectively. 
%For example, $|A| = r_{A-A} (0)$.

%
%
Given finite sets $A, B\subseteq \Gr$ and a real number $T\ge 1$ let us define a generalization of the doubling constant \eqref{def:doubling} 
\[
    \mathcal{D}[A;B] := \frac{|A+B|}{|A|}\,. 
%    \quad \quad 
\]
    In this paper we use the following new additive--combinatorial characteristics of finite sets $A,B \subseteq \Gr$
\[
    \mathcal{S}_T [A;B] = \max_{X\subseteq A,\, |X|\ge |A|/T,\, Y \subseteq B,\, |X+Y| = |X||Y|} \frac{|X||Y|}{|A|} 
\]
\[
    \le 
    \mathcal{S}[A;B] = \mathcal{S}_\infty [A;B] := \max_{X\subseteq A,\, Y \subseteq B,\, |X+Y| = |X||Y|} \frac{|X||Y|}{|A|} \le |B| \,,
 %   \quad \quad \mbox{and}
\]
    and  
    %for any $T\ge 1$ one can consider 
    %a relaxation
\[
     \mathcal{E}_T [A;B] = \max_{X\subseteq A,\, |X|\ge |A|/T,\, Y \subseteq B} \frac{|X|^2 |Y|^2}{|A|\E(X,Y)} 
\]
\[
     \le 
      \mathcal{E} [A;B]= \mathcal{E}_\infty [A;B] := \max_{X\subseteq A,\, Y \subseteq B} \frac{|X|^2 |Y|^2}{|A| \E(X,Y)} \le |B| \,.
\]
    %Then 
    Clearly, the sequences $\mathcal{S}_T [A;B]$, $\mathcal{E}_T [A;B]$   increase monotonically as $T$ tends to infinity. 
    %and 
%    and $\mathcal{E}_\infty [A;B] = \mathcal{E} [A;B]$. Moreover, 
%
%
%
If $A=B$, then we write $\mathcal{D}[A]$ for $\mathcal{D}[A;A]$, $\mathcal{S}_T [A]$ for $\mathcal{S}_T [A;A]$  and similarly for the  all other quantities.

\begin{exm}
    Let $A\subseteq \Gr$ be a Sidon set (see, e.g., \cite[Section 6.2]{TV}). 
    Then $\mathcal{D}[A] = (|A|+1)/2$ and $\mathcal{S}[A] \sim \mathcal{E}[A] \sim |A|$, since  we can take disjoint sets  $X,Y \subseteq A$, $|X+Y| = |X||Y|$ such that $|X| = |Y| = [|A|/2]$. 
\label{exm:Sidon}
\end{exm}

 In the general case, $\mathcal{D}[A;B] \neq \mathcal{D}[B;A]$, and similarly for other all quantities. 
Thus, our variation quantities given above are asymmetric.
It is easy to 
see 
%observe 
that the quantities $\mathcal{D}[A]$ and $\mathcal{S}[A]$ are different, 
consult, for example,  
%see 
%Example \ref{exm:random}. 
Proposition \ref{p:random_S} below. 
Finally, notice that if $A'\subseteq A$, $B'\subseteq B$ and $R:=|A|/|A'| \ge 1$, then by definition 
\begin{equation}\label{f:S_decrease}
    \mathcal{S}_T [A;B] \ge \mathcal{S}_T [A;B'],\,
    \quad \quad 
        \mbox{and}
    \quad \quad 
    R \cdot \mathcal{S}_{R T} [A;B] \ge \mathcal{S}_{T} [A';B]  \,.
\end{equation}
%    provided $T_1 \ge T_2 R$. 

%Finally, 
Recall that a set $\Lambda \subseteq \Gr$ is called {\it dissociated} if any  equality of the form 
\[
    \sum_{\la \in \Lambda} \eps_\la \la = 0 \,,
    \quad \quad 
    \mbox{ where }
    \quad \quad 
    \eps_\la \in \{ 0, \pm 1\} \,, \quad \quad  \forall \la \in \Lambda 
\]
implies $\eps_\la = 0$ for all $\la$. 
Let $\dim (A)$ be the size of the largest dissociated subset of $A$ and we call $\dim(A)$ the {\it additive dimension} of $A$.

The signs $\ll$ and $\gg$ are the usual Vinogradov symbols. 
If $a\ll b$ and $b\ll a$, then we write $a\sim b$. 
%When the constants in the signs  depend on a parameter $M$, we write $\ll_M$ and $\gg_M$.
%Let us denote by $[n]$ the set $\{1,2,\dots, n\}$.
All logarithms are to base $e$.
%If we have a set $A$, then we will write $a \lesssim b$ or $b \gtrsim a$ if $a = O(b \cdot \log^c |A|)$, $c>0$.
For a prime number $p$ we write $\F_p = \Z/p\Z$.
%and $\F^*_p = \F_p \setminus \{0\}$. 

\section{Preliminary results}
\label{sec:preliminary}

We start this section discussing some preliminary results on the quantities  $\mathcal{S}[A]$, $\mathcal{E}[A]$.
First of all, we need \cite[Theorem 1.3]{DSSS_co-Sidon} which was proved by a clever random choice. 
%random choice with alternation. 

\begin{theorem}
    Let $A,B \subseteq \Gr$ be sets, and $E:= \E(A,B) - |A||B|$.
    Suppose that $E>|A|$ and for some positive integers  $k,l$  such that $1\le k \le |A|/2$, $1\le l \le |B|$ one has 
\begin{equation}\label{f:E_shifts}
    kl^2 \le \frac{|A|^2 |B|^2}{2E} \,.
\end{equation}
    Then there exists sets $A' \subseteq A$, $B' \subseteq B$, $|A'| = k$, $|B'|=l$ and $|A'+B'| = |A'||B'|$. 
\label{t:E_shifts}
\end{theorem}

%Notice that for large $\E(A,B)$ our simple lemma gives the same bound for $\mathcal{D}[A;B]$ as \cite[Theorem 1.3]{DSSS_co-Sidon} but the last result has deal with a much larger ensemble of subsets of $A$ and $B$. 

Also, we need a simple lemma which  gives another connection between the quantity $\mathcal{S}[A;B]$ and the common additive energy $\E(A,B)$.

\begin{lemma}
    Let $A,B \subseteq \Gr$ be sets and $\kappa \in (0,1)$ be a real number.
    %, $\kappa \ge |B|^{-3}$. 
    Suppose that $\E(A,B) = |A||B|^2/K$.
    %, where $K \le |B|^4$. 
    Then there exists  $A_* \subseteq A$, $|A_*| > (1-\kappa)|A|$ and $B_* \subseteq B$, $|B_*| \gg (\kappa K)^{1/4}$  such that the sum
    $A_* + B_*$ is direct. 
\label{l:E_shifts}
\end{lemma}
\begin{proof}
    Let $1\le M \le K$ be a parameter which we will choose later.
    %such that $2$
    Put $$\Omega_M = \{ b,b'\in B, b\neq b' ~:~ |(A+b) \cap (A+b')| \ge |A|/M \} \,.$$
    We have $\E(A,B) = \sum_{b,b'\in B} |(A+b) \cap (A+b')|$ and hence the size of $\Omega_M$ does not exceed $M|B|^2/K$ due to
\[
    \frac{|A|}{M} \cdot |\Omega_M| \le \E(A,B) - |A| |B| < \frac{|A| |B|^2}{K} \,.
\]
%    Split the set $B$ into sets of the cardinality $$
    Now choose a set $B_* \subseteq B$ uniformly at random with probability $\d = \sqrt{K/2M} \cdot |B|^{-1}$. 
    Then taking the expectation $\mathbb{E}$, we obtain 
\[
    \mathbb{E} |\Omega_M \cap (B_* \times B_*)| = \sum_{\a \neq \beta} \Omega_M (\a,\beta) \mathbb{E} B_* (\a) B_* (\beta) = \d^2 |\Omega_M| \,,
\]
    and therefore with high probability there is a set $B_* \subseteq B$ with $2\d |B| \ge |B_*| \ge 2^{-1}\d |B| \gg \sqrt{K/M}$ and such that $\d^2 |\Omega_M| < 1$. 
    Now let 
\begin{equation}\label{def:A_*}
    A_* = A \setminus \left( \bigcup_{b,b'\in B_*,\, b\neq b'} (A-b+b') \right)
\end{equation}
and by our construction we have
\begin{equation}\label{f:|A|/2_B}
    |A_*| > |A| - \frac{|B_*-B_*| |A|}{M} 
    \ge 
    |A| - \frac{|B_*|^2 |A|}{M}
    \ge |A| - \frac{4 K}{2M^2} \cdot |A| \ge (1-\kappa) |A|  \,,
\end{equation}
    where we have chosen $M = 2\sqrt{K/\kappa}$. 
    Finally, it is easy to see that the sum $A_* + B_*$ is direct. 
    It remains to check that $2\d |B| = 2\sqrt{K/2M} = (\kappa K)^{1/4} \le |B|$ but %it follows 
    the last follows 
    from the fact that $K\le |B|$. %and our condition $\kappa \ge |B|^{-3}$. 
    %our assumption $K\le |B|^4$. 
This completes the proof. 
$\hfill\Box$
\end{proof}

\bp

%Lemma \ref{l:E_shifts} 
Theorem \ref{t:E_shifts} and Lemma \ref{l:E_shifts} (applicable in the case $E\le |A|$), as well as 
%and 
the definition of the quantity $\mathcal{E}_1 [A;B]$ 
%mean, 
give us, 
in particular, that for any $A,B \subseteq \Gr$ %with $\E(A,B) = |A||B|^2/K$ 
%and an arbitrary $T\ge 1$ 
one has  
\begin{equation}\label{f:E_S}
    \E(A,B) \gg \frac{|A||B|^2}{\mathcal{S}^2_1 [A;B]}    
    \quad \quad 
        \mbox{and}
    \quad \quad 
    \E(A,B) \ge \frac{|A||B|^2}{\mathcal{E}_1 [A;B]}
    \,,
\end{equation}
and thanks to the second formula of \eqref{f:S_decrease}, 
%we have 
we know that 
for any $A'\subseteq A$ 
%and $\kappa \in (0,1)$ 
%one has 
the following holds 
\begin{equation}\label{f:E_S_A'}
   \E(A',B) \gg \frac{|A'|^3 |B|^2}{|A|^2 \mathcal{S}^2_T [A;B]}     
   \,,
\end{equation}
    where $T=\frac{|A|}{|A'|}$.

\bp 

Let us derive another simple 
%corollary 
consequence 
of Lemma \ref{l:E_shifts}.
Given a set $A\subseteq \Gr$ consider the quantity 
\[
    \mathcal{K} (A) = \min_{s \neq 0} \frac{|A|}{|A\cap (A-s)|} \,.
\]
Also, let $k(A)$ be the length of the maximal arithmetic progression in $A$. In the next corollary we show that $\mathcal{S}[A]$ controls either $k(A)$ or $\mathcal{K} (A)$. %In other words, the lower bound for $\mathcal{S}[A]$ is $\Omega ( \min \{ k(A), \mathcal{K} (A) \})$.

\begin{corollary}
    Let $A\subseteq \Gr$ be a set.
    %$\mathcal{K} (A) = \min_{s \neq 0} \frac{|A|}{|A\cap (A-s)|}$ and 
    Then either $k(A) \le 2\mathcal{S}[A]$ or 
    $\mathcal{K} (A) \le 8 \mathcal{S}[A]$. 
    %does not exceed 
%\[
%    \mathcal{S}[A] \ge 2^{-1} \min \{ k(A) \} \,.
%    2\max \{ \mathcal{S}[A], 2 \mathcal{K} (A) \} \,.
%\]
    In other words, 
\[
    \mathcal{S}[A] \gg \Omega ( \min \{ k(A), \mathcal{K} (A) \}) \,.
\]
\end{corollary}
\begin{proof}
    Let $k = k(A)$, $P \subseteq A$ be any arithmetic progression of length $l\le k$ 
%    be the length of the maximal arithmetic progression $P$ in $A$ 
    and $K=\mathcal{K} (A)$. 
    By  definition, we have $|A\cap (A-s)|\le |A|/K$ for all $s\neq 0$. 
    Repeating the calculations in \eqref{f:|A|/2_B} (the set $A_*$ is defined in \eqref{def:A_*}), we see that 
\[
    \mathcal{S}[A] \ge \frac{|A_*| |P|}{|A|} \ge \frac{l}{2}  \,,
\]
    provided $|P-P| = 2l-1 \le K/2$. 
    Hence either $l \le 2\mathcal{S}[A]$ or $2l-1 > K/2$. 
    Thus either $k(A) \le 2\mathcal{S}[A]$ or one  can take $l=2\mathcal{S}[A]+1$ and then obtain $K\le 8 \mathcal{S}[A]$. 
This completes the proof. 
$\hfill\Box$
\end{proof}

\bp 

\begin{comment}
\begin{exm}\label{exm:random}
The lemma above is optimal, take $A=B \subseteq \F_p$ be a perfect difference set. 
%Sidon set. 
%Of course,  lemma above is optimal, just consider a random set $A=B \subseteq [N]$ such that any element of $A$ is taken with probability $1/K$. 
Further, if one takes a random set $A=B \subseteq \{1,\dots, N\}$ such that any element of $A$ is 
%taken 
chosen 
with probability $1/K$, then the bound of Lemma \ref{l:E_shifts} can be improved for small $K$. 
\end{exm}
\end{comment}

It is known that sets with small doubling have small additive dimension, see \cite{Sanders_sh}. 
We show that the same is true for sets 
%with 
having 
small quantities $\mathcal{S} [A]$ and $\mathcal{E} [A]$.

\begin{corollary}
    Then 
\[
    \dim (A) \ll \mathcal{S}^2 [A] \log |A| \,,   
    \quad \quad 
    \mbox{and}
    \quad \quad 
    \dim (A) \ll \mathcal{E} [A] \log |A|
\]
\end{corollary}
\begin{proof}
    Let $M_1 = \mathcal{E} [A]$ and $M_2 = \mathcal{S} [A]$. 
    %Indeed, it remains to notice that 
    By definition we have for all sets $X,Y \subseteq A$ 
    %one has 
\begin{equation}\label{tmp:07.08_1}
    \E(X,Y) \ge \frac{|X|^2 |Y|^2}{M_1 |A|} \,, 
\end{equation}
and for an arbitrary $Z\subseteq A$ 
%(thanks to Lemma \ref{l:E_shifts}) 
(thanks to the first formula of \eqref{f:E_S}) 
the following holds 
\begin{equation}\label{tmp:07.08_2}
     \E(A,Z) \gg \frac{|A||Z|^2}{M^2_2} \,.
\end{equation}
    Let us consider inequality \eqref{tmp:07.08_1}, as for \eqref{tmp:07.08_2} the argument is similar. 
    Take a dissociated set $\Lambda \subseteq A$ such that  $|\Lambda| = \dim (A)$ and apply the previous inequalities with $X=Z=\Lambda$ and $Y=A$. Using the H\"older inequality, as well the Rudin  bound (see, e.g., \cite[Section 4]{TV}), namely, $\T_k (\Lambda) \le k^{O(k)} |\Lambda|^k$ for any $k\ge 2$, we obtain (see details in \cite{Sanders_sh}) 
\[
    \frac{|\Lambda|^2 |A|}{M_1} \le \E(\Lambda,A) \le \T_k (\Lambda)^{1/k} |A|^{1+1/k} \ll k |\Lambda| |A|^{1+1/k} 
    \ll |\Lambda| |A| \log |A| 
    %\,,
\]
    as we can choose $k \sim \log |A|$. 
    %the required result.
This completes the proof. 
$\hfill\Box$
\end{proof}

\bp

%Now let us compute the quantity $\mathcal{S}[A]$  for a random set. 
We conclude this section by computing  the value of $\mathcal{S}[A]$  for a random set. 
The argument of the proof of Lemma \ref{l:E_shifts} suggests that this  quantity should not be as large as for Sidon sets, see Example \ref{exm:Sidon}.

\begin{proposition}
    Let $\Gr$ be a finite abelian group, $|\Gr| = N$, where  $N$ be a sufficiently large number. 
    Suppose that $A\subseteq \Gr$ be a random set formed by %selecting 
    choosing 
    elements for $A$ from $\Gr$ independently with probability $\d \in (0,1)$, $\d \gg N^{-1/2}$. 
    Then 
\begin{equation}\label{f:random_S}
    \sqrt{\frac{1}{\d}} \ll \mathcal{S}[A] 
    %\ll \log^2 |A| \cdot \sqrt{\frac{1}{\d}} \,.
    \ll 
    \sqrt{\frac{1}{\d}} \cdot \log^2 |A| \,.
\end{equation}
    If $\Gr = \Z/N\Z$, then, in addition, 
\begin{equation}\label{f:random_S_2}
    \mathcal{S}[A]  \gg \min \left\{ \frac{1}{\d}, \frac{\log |A|}{\log (1/\d)} \right\} \,.
\end{equation}
\label{p:random_S}
\end{proposition}
\begin{proof}
    Let $L=\log |A|$. 
    We follow the argument of the proof of \cite[Lemma 4.3, Theorem 1.5]{DSSS_co-Sidon}. 
    By the Chernoff's inequality we have $||A| - \d N| \le 4\sqrt{|A|}$ with probability at least $3/4$, say.
    Further, writing $f(x) = A(x) - \d$, $\mathbb{E} f = 0$, we see that 
\[
%    \E(A) = \d^4 N^3 + \E(f)
    \E(A) = \frac{|A|^4}{N} + \E(f) - \frac{(|A|-\d N)^4}{N} 
    = 
    \frac{|A|^4}{N} + \E(f) + \mathcal{E} \,, 
\]
where $|\mathcal{E}| \le 2^8 |A|^2 N^{-1}$. 
Hence 
\[
     \mathbb{E} \E(A) = \d^4 N^3 + \mathbb{E} %\sum_{x+y=z+w~:~ x,y,z,w \mbox{ are not distinct} } 
     \sum_{x+y=z+w~:~ |\{x,y,z,w \}|\le 2} 
     f(x)f(y)f(z)f(w) + \mathcal{E} = \d^4 N^3 +  \mathbb{E} Z + \mathcal{E} \,,
\]
    where $\mathbb{E} Z \le 3|A|^2 \le 6 \d^2 N^2$. 
    Similarly, as $|f(x)|^2 \le A(x) + \d^2$, we have 
\[
    \mathbb{E} Z^2 \le 3\mathbb{E} \sum_{x_1,\dots,x_4} (A(x_1) + \d^2) \dots (A(x_4) + \d^2) \le 48 |A|^4 \,.
\]
Thus with probability at least $1/2$ we have $|Z| \le 10 \d^2 N^2$. 
In particular, one has $\E(A) \ll \d^4 N^3 + \d^2 N^2$ and by Theorem \ref{t:E_shifts} 
(see the first estimate in \eqref{f:E_S}) 
%one has 
the following holds 
%By Lemma \ref{l:E_shifts} we know that 
$\mathcal{S}[A] \gg \sqrt{\frac{1}{\d}}$.
%After that 
Now to obtain the upper bound from \eqref{f:random_S}  we apply the argument of the proof of \cite[Lemma 4.3]{DSSS_co-Sidon} and derive that 
\[
\max_{x\neq 0}\, r_{A-A}(x) \le 2\d^2 N + 3\log N\]
with probability greater than $9/10$ under the assumption $\d \ge N^{-2/3}$. 
Finally, repeating the proof of \cite[Theorem 1.5]{DSSS_co-Sidon} and using the obtained bound $|Z| \le 10 \d^2 N^2$, we see that with probability $1-o(1)$ one has 
\begin{equation}\label{f:kl^2}
    k l^2 \ll L^2 N  
\end{equation}
    for any $X,Y\subseteq A$, $|X|=k$, $|Y|=l$, $k\ge 2l$ and $|X+Y| = |X||Y|$.  
    Now suppose that $\mathcal{S}[A]$  is attained at $X,Y\subseteq A$ with $|X+Y| = |X||Y|$, $|X|=k$, $|Y|=l$ and such that  $k\ge 2l$ (taking a half of $X$ or $Y$ we can always assume that the last condition holds). 
    Combining the lower bound \eqref{f:random_S} and estimate \eqref{f:kl^2}, we get 
\[
    \d^{-1/2} \ll \mathcal{S}[A] \ll  \frac{kl}{|A|} \ll \frac{L^2}{\d l}
\]
and hence $l \ll L^2 \d^{-1/2}$. 
But then 
\[
    \mathcal{S}[A] \ll \frac{kl}{|A|} \le l \ll L^2 \d^{-1/2}
\]
as required.

To get \eqref{f:random_S_2}, simply  find an arithmetic progression $P\subseteq A$ of size $\Omega(L/ \log(1/\d))$ (it is possible to do with high probability). Also, we can assume that $\d \gg \sqrt{L/N}$ (otherwise estimate \eqref{f:random_S_2} is trivial) and hence 
$$
    \max_{x\neq 0}\, r_{A-A}(x) \le 2\d^2 N + 3\log N \ll \d^2 N \,.
$$
Using the argument of the proof of Lemma \ref{l:E_shifts} with $B=P$, we see that $\mathcal{S}[A] \gg \d^{-1}$, provided $|B| \ll \d^{-1}$. 
This completes the proof. 
$\hfill\Box$
\end{proof}

%\bp 

\section{The proof of the main result}
\label{sec:proof}

Now we are ready to prove our main result which shows that all quantities $\mathcal{S} [A]$, $\mathcal{E} [A]$ and $\mathcal{D} [A]$ are polynomial equivalent. 
Additionally, the estimate \eqref{f:equiv_n,m} below says that the size of the maximum direct sum in $A$ and $nA$ are somewhat comparable.
%Also, bound \eqref{f:equiv_n,m} below says that the size of maximal direct sum in $A$ and $nA$ are somewhat comparable. 

\begin{theorem}
%    All quantities above $$
%    One has 
    Let $\Gr$ be an abelian group, and $A,B \subseteq \Gr$ be sets. 
    Then for any $T\ge 1$ the following holds 
\begin{equation}\label{f:equiv}
    1\le \mathcal{S}_T [A;B]
    %, \mathcal{E}_T [A;B]  
    \le \mathcal{E}_T [A;B] \le \mathcal{D}[A;B], 
    %\ll \mathcal{S}^C [A] 
    \quad \quad 
    \mathcal{E}_T [A;B] \ll T \mathcal{S}^2_T [A;B]
    \,,
\end{equation}
    and for all $T\ge 2$ there are positive constants $C_1,C_2>0$ such that 
\begin{equation}\label{f:equiv_sym}
    \mathcal{D}[A] \ll \mathcal{S}^{C_1} [A] \,,
    \quad \quad 
    \mbox{and}
    \quad \quad 
    \mathcal{D}[A] \ll \mathcal{E}^{C_2}_T [A] \,.
\end{equation}
    In particular, for any non--zero integers $n$ and $m$ one has 
\begin{equation}\label{f:equiv_n,m}
    \mathcal{S}[nA;mA] \le \mathcal{E}[nA;mA] \le \mathcal{D}[nA;mA] \ll 
    %\mathcal{S}^{C\max \{|n|,|m|\}} [A] 
    \mathcal{S}^{C_1 (|n| + |m|)} [A] 
    \,.
\end{equation}  
\label{t:equiv}
\end{theorem}
\begin{proof}
    Let $X=A$ and let $Y \subseteq B$ be any  set such that $|Y| =1$. Then we see that $\mathcal{S}_T [A;B]\ge 1$. 
    Now let us show that $\mathcal{S}_T [A;B] \le \mathcal{E}_T [A;B]$.
    Indeed, find $X\subseteq A$, $|X|\ge |A|/T$ and $Y \subseteq B$ such that $|X+Y|= |X||Y|$ and $\mathcal{S}_T [A;B] = \frac{|X||Y|}{|A|}$.
    Then 
\[
    \mathcal{E}_T [A;B] \ge \frac{|X|^2 |Y|^2}{|A| \E(X,Y)} = \frac{|X||Y|}{|A|} = \mathcal{S}_T [A;B]
\]
    as required. 
    Further, by the Cauchy--Schwarz inequality \eqref{f:energy_CS}, one has  $\E(X,Y) \ge \frac{|X|^2 |Y|^2}{|A+B|}$ and hence for any $T\ge 1$ the following holds $\mathcal{E}_T [A;B] \le \mathcal{D}[A;B]$.
    Finally, by 
    %the second 
    formulae in \eqref{f:S_decrease}, \eqref{f:E_S}, we have (recall that $|X|\ge |A|/T$)
%imply that 
\[
    \mathcal{E}_T [A;B] = \frac{|X|^2 |Y|^2}{|A| \E(X,Y)}
    \ll \frac{|X|}{|A|} \cdot \mathcal{S}^2_1 [X;Y] 
    \le 
    \frac{|X|}{|A|} \cdot \mathcal{S}^2_1 [X;B] 
    \le 
    T \mathcal{S}^2_T [A;B] \,.
\]
    In particular, for any fixed $T$ the quantities $\mathcal{E}_T [A;B]$ and $\mathcal{S}_T [A;B]$ are polynomial equivalent.

    Now let $M=\mathcal{S}[A;B]$ and $\E(A,B) = |A| |B|^2/K$. 
    Applying the first estimate of \eqref{f:E_S} %%with $A=B$, 
    %Lemma \ref{l:E_shifts} and using the definition of the quantity $\mathcal{S}[A;B]$ (or just consult the second bound of \eqref{f:E_S})
    we see that 
\begin{equation}\label{f:M_K_connection}
    \E(A,B) = \frac{|A| |B|^2}{K} \gg \frac{|A| |B|^2}{M^2} \,,
%    M |A| \ge |A_*| |B_*| \gg |A| \sqrt{K} \,,
\end{equation}
    and therefore  $K \ll M^2$. 
    Hence the common additive energy $\E(A,B)$ is large.
    Recall that in our case $A=B$  and we need to show that 
    %let us show that 
    $\mathcal{D}[A] \ll \mathcal{S}^C [A]$.
%    First of all consider the case $A=B$ and show that  $\mathcal{D}[A] \ll \mathcal{S}^C [A]$.  
    %Applying a variant 
    Using a version 
    of the Balog--Szemer\'edi--Gowers Theorem \ref{t:BSzG_intr}, we find $A'\subseteq A$, $|A'| \gg |A|/K$ such that $|A'-A'| \ll K^C |A'|$, where $C>0$ is an absolute constant (one can take $C=4$, see, e.g., \cite{Schoen_BSzG}). 
    Consider a maximal set $Y \subseteq A$ such that the sum $A'+Y$ is disjoint. Then, by maximality, one has $A\subseteq A'-A'+Y$ and therefore by the Pl\"unnecke inequality \eqref{f:Pl-R} (also, see \cite{Petridis}) one has 
\begin{equation}\label{f:A+A_Y}
    |A+A| \le |2A'-2A'+Y+Y| \le |2A'-2A'| |Y|^2  \ll K^{4C} |Y|^2 |A'|
    \ll
    M^{8C}  |Y|^2 |A'| \,.
\end{equation}
    It remains to estimate $|Y|$. 
    But using  the definition of the quantity $\mathcal{S} [A]$ one more time, we get 
$
    M |A| \ge |A'| |Y|  
$ 
    and thus 
\[
    \mathcal{D}[A] |A|^2 M^{-2} \ll |A+A| |A'| \ll M^{8C+2} |A|^2  
    %\,.
\]
    as required.
    Finally, the first estimate in \eqref{f:equiv_sym} immediately implies \eqref{f:equiv_n,m}. 
%    Indeed, without loosing of the generality assume that  $n=\max\{|n|,|m|\} >0$. 
%    We have 
    Indeed, by the Pl\"unnecke inequality \eqref{f:Pl-R}, we obtain 
\[
    \mathcal{S}[nA;mA] \le \mathcal{D}[nA;mA] \le \mathcal{D}^{|n|+|m|}[A] \ll \mathcal{S}^{C_1 (|n|+|m|)}[A] \,.
\]

%    It remains to prove 
    Now let us obtain the second bound from \eqref{f:equiv_sym}, 
    and it suffices to consider only the case $T=2$.
    %and it is enough to consider the case $T=2$ only. 
    Let $\mathcal{E}_2 [A] = M_*$. 
    Then $\E(A) \ge |A|^3/M_*$ and applying the Balog--Szemer\'edi--Gowers Theorem again, we find $A_1 \subseteq A$, $|A_1| \gg |A|/M_*$ such that $|A_1-A_1| \ll M^C_* |A_1|$, where $C>0$. 
    If $|A_1|\ge |A|/2$, then we can use the argument as above to find  a maximal set $X\subseteq A$ such that the sum $A_1+X$ is direct, hence  $A\subseteq A_1-A_1+X$ and we obtain an analogue of \eqref{f:A+A_Y} with $K=M_*$.
    Also, by definition of the quantity $\mathcal{E}_2 [A]$, we have 
\[
    |A_1||X| = \E(A_1,X) \ge \frac{|A_1|^2 |X|^2}{M_* |A|} \ge \frac{|A_1||X|^2}{2M_*} \,,
\]
    and hence $|X| \le 2M_*$. Thus  
    $\mathcal{D}[A] \ll M^{4C+2}_*$. 
    Now if $|A_1|\le |A|/2$, then we take $A^c_1 = A\setminus A_1$, $|A^c_1|\ge |A|/2$ and by definition of the quantity $\mathcal{E}_2 [A]$, one has 
    %we have
\begin{equation}\label{tmp:7.08_3}
    \sum_{x} |A^c_1 \cap (A_1+x)|^2 = \E(A^c_1,A_1) \ge \frac{|A^c_1||A_1|^2}{2M_*} \,.
\end{equation}
    Hence there is $x_1\in \Gr$ such that $|A^c_1 \cap (A_1+x_1)| \ge |A_1|/(2M_*) \gg |A|/M^2_*$. 
    Put $A_2 = A_1 \bigsqcup (A^c_1 \cap (A_1+x_1))$ and then consider $A^c_2 = A\setminus A_2$. 
    If $|A_2| \ge |A|/2$, then we stop our algorithm. 
    If not, then we have an analogue of \eqref{tmp:7.08_3}, namely, 
\[
    \sum_{x} |A^c_2 \cap (A_1+x)|^2 = \E(A^c_2,A_1) \ge \frac{|A^c_2||A_1|^2}{2M_*} \,,
\]
    and therefore exists $x_2\in \Gr$ such that $|A^c_2 \cap (A_1+x_2)| \ge |A_1|/M_* \gg |A|/M^2_*$. 
    Put $A_3 = A_2 \bigsqcup (A^c_2 \cap (A_1+x_2))$ and consider $A^c_3 = A\setminus A_3$. 
    And so on. Clearly, our algorithm must stop after 
    %at most 
    no more than 
    $s=O(M^2_*)$ 
    %number of 
    steps. 
    At the end we have a set $A' \subseteq A\cap (A_1 + Z)$, $|Z| \le s$, $|A'|\ge |A|/2$. 
    Applying the argument as above, we find $W\subseteq A$, $|W| \le 2M_*$ and such that  $A \subseteq A'-A'+W$. 
    It gives us 
\[
    |A+A| \le |2A'-2A'+2W|\le |2A_1-2A_1+2W+2Z-2Z| \le |2A_1-2A_1| |Z|^4 |W|^2 \ll M^{4C+10}_* |A_1| \,,
\]
    and hence  $\mathcal{D}[A] \ll M^{4C+10}_*$. 
%    It remains to prove 
This completes the proof. 
$\hfill\Box$
\end{proof}

\bp 

As one can see from the proof it is possible to take $C_1= 36$ and $C_2=26$ in \eqref{f:equiv_sym}.
Now let us make a series of remarks concerning possible generalizations and limitations of Theorem \ref{t:equiv}.

%\bp 

%Also, given $\a \in [0,1]$, define 
%\[
%    \mathcal{R} [A] = \max_{A_* \subseteq A,\, |A_*|\ge |A|/2}\, \max_{A' \subseteq A} R_{A'} [A_*] 
%    \mathcal{R} [A] = \max_{A_* \subseteq A,\, |A_*|\ge \a|A|} R_{A} [A_*] \ge R_{A} [A] \,.
%\]

\begin{remark}
It is easy to see that 
%in the case $A=B$ 
the quantity 
\[
    %R_B [A] 
    R_A [A] 
    = \min_{\emptyset \neq X \subseteq A} \frac{|B+X|}{|X|} 
\]
appeared in the modern proof of the Pl\"unnecke inequality \eqref{f:Pl-R} (see \cite{Petridis}) 
is polynomial equivalent to the doubling constant $\mathcal{D} [A]$  and hence  $\mathcal{S} [A]$. 
Indeed, clearly,  
%$R_B [A] \le \mathcal{D} [A;B]$
$R_A [A] \le \mathcal{D} [A]$ 
but from the Petridis lemma \cite{Petridis} one has 
\[
    \mathcal{D} [A] |A| = |A+A| \le |A+A+X| \le R^2_A [A] |A|  \,,
\]
and thus $\mathcal{D} [A] \le R^2_A [A]$. 
%On the other hand, it is easy to see that in the asymmetric case the quantities one has $R_B [A]$, $\mathcal{D} [A;B]$ are not polynomial equivalent, although $R_B [A] \le \mathcal{D} [A;B]$.  Indeed, take $A=S+P$, $B=T+P$, where all sums are direct and $T,P$ are arithmetic progressions (or subgroups). Then $\mathcal{D} [A;B] \sim |T|$ but $R_B [A] \le \frac{|P+T+T|}{|T|} \ll |P|$ and hence $R_B [A]$, $\mathcal{D} [A;B]$ are incomparable. 
\end{remark}

\begin{remark}
    It is easy to see that quantities   
$\mathcal{S}[A;B], \mathcal{E}[A;B]$ and $\mathcal{D}[A;B]$ are incomparable. 
Indeed, let $\Gr = \F_2^n$, $k$ be a large integer parameter, $A = \bigsqcup_{j=1}^k H_j$, where $H_j \le \Gr$ are mutually additively disjoint, having the same size $h$,  and $B= \bigsqcup_{j=1}^k S_j$, where $S_j \subseteq H_j$ are some sets, $|S_j| \sim \sqrt{|H_j|}$, say. 
Then $|A| = hk$, $|B| \sim k \sqrt{h}$ and 
$\mathcal{D}[A;B]$ is $\Omega(hk)$ but 
$\mathcal{S}[A;B] \ll |B|^2/|A| = k$ is incomparable with $hk$. A similar calculation shows 
%(or just see estimate \eqref{f:equiv} 
(or just see the next remark) 
that $\mathcal{E}[A;B] \sim k^2$.

Nevertheless,  it is easy to see that any two sets  $A, B \subseteq \Gr$, $L:= |A|/|B| \ge 1$ contain pairs of large subsets $(H_j,S_j)$, $j\in [k]$ such that $\mathcal{D}[H_j;S_j]$ and $\mathcal{S}[H_j;S_j]$ are polynomial  comparable. Here we give a sketch of the proof. Let $M= \mathcal{S}[A;B]$. Then as in \eqref{f:M_K_connection} one derives $\E(A,B)\ge |A||B|^2/M^2$. After that apply the asymmetric Balog--Szemer\'edi--Gowers theorem \cite[Theorem 2.35]{TV} and find a subset $B'= B\cap H$, as well as sets $H,\Lambda \subseteq \Gr$ such that $|B'|\gg M^{-O_\eps (1)} L^{-\eps} |B|$, $|H+H|\ll M^{O_\eps (1)} L^{\eps} |H|$ and $A' = A\cap (H\dotplus \Lambda)$, $|A'|\gg M^{-O_\eps (1)} L^{-\eps} |A|$, 
%and 
as well as 
$|\Lambda| |H| \ll M^{O_\eps (1)} L^{\eps} |A|$. 
Then it is easy to see that $\mathcal{D}[A';B'] \sim \mathcal{S}[A';B'] \sim M^{O_\eps (1)} L^{\eps}$. 
\end{remark}

\begin{remark}
%    As we have see 
    Theorem \ref{t:equiv} says, in particular, 
    %to us 
    that $\mathcal{E}_{2} [A]$ and $\mathcal{D} [A]$ are polynomial dependent. 
    Nevertheless, one can show that 
    %for any fixed $M$
    the quantities  $\mathcal{E}_1 [A]$ and $\mathcal{D} [A]$ are incomparable.
    Indeed, let $\Gr = \F_2^n$, $k$ be a large integer parameter, $A = \bigsqcup_{j=1}^k H_j$, where $H_j \le \Gr$ are mutually additively disjoint and having the same size $h$. Then $\mathcal{D}[A] \sim |A|$ 
    %for any $A'\subseteq A$, $|A'|\ge |A|/M$, $A = \bigsqcup_{j=1}^k (A\cap H_j) = \bigsqcup_{j=1}^k A_j$ 
    and for an arbitrary $X\subseteq A$, $X = \bigsqcup_{j=1}^k (X\cap H_j) = \bigsqcup_{j=1}^k X_j$ one has 
\[
    \E(A,X) \ge \sum_{j=1}^k \E(A_j,X_j) \ge h^{} \sum_{j=1}^k |X_j|^2 \ge k^{-1} h |X|^2  = k^{-2} |A| |X|^2 \,, 
\]
    and hence $\mathcal{E}_1 [A] \le k^2$. 
\end{remark}

%As we have seen in the 
%%Remark \ref{}
%remark above the quantities $\mathcal{S}[A;B]$ and $\mathcal{D}[A;B]$ are incomparable and hence the condition that $\mathcal{S}[A;B]$ is small is rather restrictive. Below we give a full description of such sets. 

%\bp 

%In 
Once again, 
the second formula of \eqref{f:equiv_sym} from  Theorem \ref{t:equiv} says that $\mathcal{E}^{}_{2} [A]$ and  $\mathcal{D}^{} [A]$ are equivalent. 
We 
%will 
now show that in some sense similar results can be obtained for $\mathcal{S}^{}_{2+o(1)} [A]$.
%Now we show that one can obtain a similar results for $\mathcal{S}^{}_{2+o(1)} [A]$, in a sense. 

\begin{corollary}
     Let $\Gr$ be an abelian group, and $A \subseteq \Gr$ be a set. Suppose that $|A+A|=K|A|$ and $\kappa \in (0,1/2)$ is a real number.  
     Then there is an absolute constant $c>0$ and sets $X,Y \subseteq A$ such that 
%     $|X| \ge |A|/4$, $|X||Y| \gg |A|$ and $|X+Y| = |X| |Y|$
\[
    |X| \ge (1/2-\kappa) |A|,  \quad |X+Y| = |X| |Y| \quad \mbox{and} \quad |X||Y| \gg \kappa^{1/4} 
    %\cdot 
    K^c |A| \,.
\]
\label{c:0.5}
\end{corollary}
\begin{proof}
    By Theorem \ref{t:equiv} we know that $K:=\mathcal{D}[A]$ and $M:=\mathcal{E}^{}_{2} [A]$ are polynomial equivalent, that is $K\ll M^C$, where $C>0$ is an absolute constant.
    It means that $\E(X_*,Y_*) \gg \frac{|X_*||Y_*|^2}{M}$ for some $X_*,Y_* \subseteq A$, $|X_*|\ge |A|/2$. 
    Using the argument of the proof of Lemma \ref{l:E_shifts}, we find some sets $X\subseteq X_*$, $Y \subseteq Y_*$ such that $|X+Y| = |X| |Y|$, $|X| \ge (1/2-\kappa) |A|$ and $|Y|^4 \sim \kappa M$
    as required. 
    Another possible way to obtain the same is to repeat the exhaustion argument of the proof of Theorem \ref{t:equiv}. 
%    By the argument of the proof of Theorem \ref{t:equiv} we now that $\mathcal{D}[A] =K$ and $M:=\mathcal{E}^{}_{(1-\kappa)^{-1}} [A]$ are polynomial equivalent, that is $K\ll_\kappa M^C$, where $C>0$ is an absolute constant. 
    This completes the proof. 
$\hfill\Box$
\end{proof}

\begin{remark}
\label{r:1/2-kappa}
    It is easy to see that it is impossible to replace $1/2-\kappa$ to  $1/2+\kappa$ in Corollary \ref{c:0.5}.
    Indeed, let $\Gr = \F_2^n$,  $A = H_1 \bigsqcup H_2$, where $H_1, H_2 \le \Gr$ are mutually additively disjoint and having the same size $|A|/2$. 
    Then $\mathcal{D}[A] \sim |A|$ but if $X\subseteq A$, $|X| \ge (1/2+\kappa)|A|$ then it is easy to see that any set $Y\subseteq A$ such that $|X+Y| = |X||Y|$ satisfies $|Y| \ll \kappa^{-1}$. 
    It implies that  $\mathcal{S}_{(1/2+\kappa)^{-1}}[A] \ll \kappa^{-1}$ 
    %(and hence by formula \eqref{f:equiv} automatically $\mathcal{E}_{(1/2+\kappa)^{-1}}[A] \ll \kappa^{-2}$) 
    (and, therefore, according to the formula \eqref{f:equiv} automatically $\mathcal{E}_{(1/2+\kappa)^{-1}}[A] \ll \kappa^{-2}$) 
    and thus this quantity is incomparable with $\mathcal{D}[A]$. 
\end{remark}

%\bp 

Let $A,B \subseteq \Gr$ be sets. 
Suppose that the quantity $\mathcal{S} [A;B]$ is attained at $X\subseteq A$, $Y\subseteq B$, that is $\mathcal{S} [A;B] = |X||Y|/|A|$, where $|X+Y|=|X||Y|$.
Then by the maximality one has 
\[
    B \subseteq X-X+Y
    \quad \quad 
    \mbox{and}
    \quad \quad 
    A \subseteq Y-Y+X \,.
\]
Thus, an analogue of the covering lemma (see, e.g.,  \cite[Section 2.4]{TV}) takes place (recall that the sizes of $X,Y$ are controllable, namely, $|X||Y| = |A| \mathcal{S} [A;B]$).  
Similarly, 
%it is easy to see 
we show 
that an  analogue of the triangle inequality \eqref{f:triangle_intr} holds. 
Surprisingly, but our bounds \eqref{f:triangle_S}, \eqref{f:triangle_E++}  do not depend on cardinalities of sets unlike in the classical triangle inequality. 
In the proof we use an idea from the higher energies \cite{SS_higher}.

%\bp 

\begin{theorem}
    Let $A,B,C \subseteq \Gr$ be sets, and $T\ge 1$ be a real number.
    Then 
\begin{equation}\label{f:triangle_S}
    \mathcal{S}_T [B;C] \le \mathcal{E}_T [B;A] \mathcal{S} [A;C] 
    \quad \quad 
    \mbox{and}
    \quad \quad 
    \mathcal{E}_T [B;C] \le \mathcal{E}_T [B;A] \mathcal{E} [A;C] 
    \,. 
\end{equation}
%    Further, for any $\kappa \in (0,1)$ one has 
    In particular, one has 
\begin{comment}
\begin{equation}\label{f:triangle_S+}
%    \kappa \mathcal{S}_T [B;C] \ll \mathcal{S}^2_{T (1-\kappa)^{-1}} [B;A] \mathcal{S} [A;C] \,. 
    \mathcal{S}_T [B;C] \le T \mathcal{S}^2_T [B;A] \mathcal{S} [A;C] 
    \quad \quad 
    \mbox{and}
    \quad \quad 
    \mathcal{E}_T [B;C] \le \mathcal{E}_T [B;A] \mathcal{E} [A;C] \,,
%    \,, 
\end{equation}
%    In particular, 
%\begin{equation}\label{f:triangle_S++}
%    \mathcal{S} [B;C] \ll \mathcal{S}^2_{} [B;A] \mathcal{S} [A;C] \,,
%\end{equation}
    as well as 
\end{comment}
\begin{equation}\label{f:triangle_E++}
%   \mathcal{E}_T [B;C] \ll T \mathcal{E}^2_{} [B;A] \mathcal{E}^2 [A;C] \,. 
%    \mathcal{E}_T [B;C] \le \mathcal{E}_T [B;A] \mathcal{E} [A;C] \,. 
    \mathcal{S} [B;C] \le \mathcal{E} [B;A] \mathcal{S} [A;C]
    \quad \quad 
%    \mbox{and, finally, }
    \mbox{and}
    \quad \quad 
    \mathcal{E} [B;C] \le \mathcal{E} [B;A] \mathcal{E} [A;C] \,. 
\end{equation}
\label{t:triangle_S}
\end{theorem}
\begin{proof}
    Find $X\subseteq B$, $Y\subseteq C$ such that $|X| \ge |B|/T$, $|X+Y| = |X||Y|$ and $|X||Y| = |B| \mathcal{S}_T [B;C]$. 
    We have 
\begin{equation}\label{tmp:13.08_-1}
    \sum_{z} |A\cap (X+z)|^2 = \E(X,A) \ge \frac{|X|^2 |A|^2}{\mathcal{E}_T [B;A] |B|}
\end{equation}
    and therefore there is $z$ such that the set $A':= A\cap (X+z)$ has size at least $|A||X|/(\mathcal{E}_T [B;A] |B|)$. 
    But then
\[
    \mathcal{S} [A;C] \ge \frac{|A'||Y|}{|A|} \ge \frac{|X||Y|}{\mathcal{E}_T [B;A]|B|} 
    = 
    %\frac{|B| \mathcal{S}_T [B;C]}{|X| \mathcal{E}_T [B;A]}
    %\ge 
    \frac{\mathcal{S}_T [B;C]}{\mathcal{E}_T [B;A]}
\]
    as required. 
%    Further, to obtain bounds \eqref{f:triangle_S+}, \eqref{f:triangle_E++} just apply the estimates  \eqref{f:equiv} of Theorem \ref{t:equiv}. 

    Similarly, take $X\subseteq B$, $Y\subseteq C$ such that $|X| \ge |B|/T$ and 
\begin{equation}\label{tmp:13.08_1}
    |X|^2 |Y|^2 = |B| \mathcal{E}_T [B;C] \E(X,Y) \,.
\end{equation}
    For an arbitrary $z\in \Gr$ put $A^X_z = A \cap (X+z)$. 
    Then by  definition, we have 
\[
    \E(A^X_z, Y) \cdot |A| \mathcal{E} [A;C]  \ge |A^X_z|^2 |Y|^2 \,.
\]
    Summing the last inequality over all $z$, we get 
\[
    |A| \mathcal{E} [A;C] \cdot \sum_z \E(A^X_z, Y)
    = 
    |A| \mathcal{E} [A;C] \cdot \sum_w r_{A-A} (w) r_{X-X} (w) r_{Y-Y} (w)
    \ge |Y|^2 \E(X,A) \,.
\]
    Combining this estimate, as well as formulae \eqref{tmp:13.08_-1}, \eqref{tmp:13.08_1}, we obtain 
\[
    |A|^2 \mathcal{E} [A;C] \E(X,Y) \ge |Y|^2 \frac{|X|^2 |A|^2}{\mathcal{E}_T [B;A] |B|} 
    \ge 
    \frac{|A|^2 \mathcal{E}_T [B;C] \E(X,Y) }{\mathcal{E}_T [B;A]} 
\]
    as required. 
    Finally, taking 
    $T=\infty$, we get 
    \eqref{f:triangle_E++}.
    This completes the proof. 
$\hfill\Box$
\end{proof}

\begin{comment}

\bp 

Given a set $A\subseteq \Gr$ define 
\[
    \mathcal{D}[A] = \frac{|A+A|}{|A|}\,,
    \quad \quad 
    \mathcal{S}[A] = \max_{X\subseteq A,\, Y \subseteq A,\, |X+Y| = |X||Y|} \frac{|X||Y|}{|A|}\,,
    \quad \quad 
    \mathcal{E}[A] = \max_{X\subseteq A,\, Y \subseteq A} \frac{|X|^2 |Y|^2}{|A| \E(X,Y)} \,.
\]
\end{comment}

\bibliographystyle{abbrv}

\bibliography{bibliography}{}

\noindent{I.D.~Shkredov\\
%Steklov Mathematical Institute,\\
%ul. Gubkina, 8, Moscow, Russia, 119991}
\\
%and
%\\
%IITP RAS,  \\
%Bolshoy Karetny per. 19, Moscow, Russia, 127994\\
%and 
%\\
%MIPT, \\ 
%Institutskii per. 9, Dolgoprudnii, Russia, 141701\\
{\tt ilya.shkredov@gmail.com}

\end{document}

%%%%%%%%%%%%%%%%%%

\begin{comment}    
    we use estimate \eqref{f:E_S_A'} and derive 
\[
    \sum_{z} |A\cap (X+z)|^2 = \E(X,A) \gg \frac{|X|^3 |A|^2}{|B|^2 \mathcal{S}^2_{T_*} [B;A]} \,,
\]
    where $T_* = \frac{2|B|}{|X|}$. 
    %\le T (1-\kappa)^{-1}$. 
    After that repeat the proof above and obtain 
\[
     \mathcal{S} [A;C] \ge \frac{|A'||Y|}{|A|} \gg \frac{|X|^2 |Y|}{|B|^2 \mathcal{S}^2_{T_*} [B;A]} 
     =
     \frac{|X|\mathcal{S}_T [B;C]}{|B| \mathcal{S}^2_{T_*} [B;A]} 
     \ge 
     \frac{\mathcal{S}_T [B;C]}{T\mathcal{S}^2_{T_*} [B;A]} \,.
\]
    Taking 
    %$\kappa =1/2$, we get 
    $T=\infty$ in \eqref{f:triangle_S+}, we get 
    \eqref{f:triangle_S++}. 
    Finally, using the last formula of  \eqref{f:equiv}, as well as the bound \eqref{f:triangle_S}, we obtain \eqref{f:triangle_E++}. 
    %\mathcal{E}_T [A;B] \ll T \mathcal{S}^2_T [A;B]
\end{comment}

%%%%%%%%%

    As above, we find $A':= A\cap (X+z)$ with 
    $|A'| \ge |A||X|/(\mathcal{E}_T [B;A] |B|)$  
    %$|A'|\ge |A|/\mathcal{E}_T [B;A]$ 
    and therefore thanks to formula \eqref{tmp:13.08_1}, we obtain 
\[
    \mathcal{E} [A;C] 
    \ge 
        \frac{|A'|^2 |Y|^2}{|A| \E(A',Y)} 
    \ge 
        \frac{|A'|^2 |Y|^2}{|A| \E(X,Y)}
    \ge 
        \frac{|A||X|^2 |Y|^2}{|B|^2 \mathcal{E}^2_T [B;A] \E(X,Y)}
        =
        \frac{|A|}{|B| \mathcal{E}^2_T [B;A]}   
    = \frac{|A'|^2 |B| \mathcal{E}_T [B;C]}{|A| |X|^2} 
    \ge 
%    \frac{|A| |B| \mathcal{E}_T [B;C]}{\mathcal{E}^2_T [B;A] |X|^2} 
%    \frac{|A'| |B| \mathcal{E}_T [B;C]}{\mathcal{E}_T [B;A] |X|^2} 
\]